\pdfoutput=1

\documentclass[reqno,a4paper,final]{amsart}

\usepackage[utf8]{inputenc}
\usepackage[T1]{fontenc}
\usepackage[english]{babel}
\usepackage{ifthen}
\usepackage{enumitem}
\usepackage{dsfont}
\usepackage{mathtools}
\usepackage{amssymb}
\usepackage[babel]{csquotes}
\usepackage[backend=biber,style=alphabetic]{biblatex}
\usepackage[protrusion=true,expansion=true,babel=true,final]{microtype}
\usepackage[unicode,bookmarksopen]{hyperref}

\numberwithin{equation}{section} 

\newtheorem{lemma}{Lemma}[section]

\newtheorem{proposition}[lemma]{Proposition}
\newtheorem{theorem}[lemma]{Theorem}

\theoremstyle{definition}

\newtheorem{remark}[lemma]{Remark}
\newtheorem{problem}[lemma]{Problem}

\newlist{thm_enum}{enumerate}{1}
\setlist[thm_enum]{label=\normalfont(\alph*)}
\newlist{def_enum}{enumerate}{1}
\setlist[def_enum]{label=\normalfont(\roman*)}
\newlist{equiv_enum}{enumerate}{1}
\setlist[equiv_enum]{label=\normalfont(\roman*)}

\newcommand{\IR}{\mathbb{R}}
\newcommand{\IC}{\mathbb{C}}
\newcommand{\IE}{\mathbb{E}}

\newcommand{\abs}[1]{\left\lvert#1\right\rvert}

\newcommand{\norm}[1]{\left\lVert#1\right\rVert}
\newcommand{\normalnorm}[1]{\lVert#1\rVert}

\newcommand{\R}[2][\empty]{
	\ifthenelse{\equal{#1}{\empty}}
		{\mathcal{R}\left\{#2\right\}}
		{\mathcal{R}_{#1}\left\{#2\right\}}
}

\makeatletter
\newcommand{\LeftEqNo}{\let\veqno\@@leqno}
\makeatother

\renewcommand{\d}{\mathop{}\!d}
\renewcommand{\Re}{\operatorname{Re}}

\renewcommand{\epsilon}{\varepsilon}

\let\temp\phi
\let\phi\varphi
\let\varphi\temp

\DeclareMathOperator{\E}{\IE}
\DeclareMathOperator{\linspan}{span}

\DeclareMathOperator{\Id}{Id}

\DeclareMathOperator{\MaxReg}{MR}

\allowdisplaybreaks[4]

\DeclareSourcemap{
  \maps[datatype=bibtex, overwrite]{
    \map{
      \step[fieldset=abstract, null]
      \step[fieldsource=entrykey,match=Mer99,final] %
      \step[fieldset=shorthand,fieldvalue=LeM99] %
    }
  }
}

\renewbibmacro{publisher+location+date}{%
	\printlist{publisher}
	\setunit*{\addcomma\space}
	\printlist{location}
  	\setunit*{\addcomma\space}
  	\usebibmacro{date}
\newunit} 

\newbibmacro{string+doi+url}[1]{%
	\iffieldundef{doi}{%
			\iffieldundef{url}{#1}{\href{\thefield{url}}{#1}}%
		}%
		{\href{http://dx.doi.org/\thefield{doi}}{#1}}
}%

\renewbibmacro{in:}{}

\DeclareFieldFormat*{title}{\usebibmacro{string+doi+url}{\mkbibemph{#1}}}
\DeclareFieldFormat*{booktitle}{#1}
\DeclareFieldFormat[article]{volume}{\textbf{#1}\addcomma}
\DeclareFieldFormat[article]{number}{\addspace no.~#1}
\DeclareFieldFormat[article]{pages}{#1}
\DeclareFieldFormat{journaltitle}{#1} 
\DeclareFieldFormat{url}{}
\DeclareFieldFormat{eprint}{Preprint on arXiv: \href{http://arxiv.org/abs/#1}{#1}} 

\renewcommand*{\bibfont}{\footnotesize}

\begin{document}

\title[Counterexamples to Non-Autonomous Maximal Regularity]{J.-L.~Lions' Problem Concerning Maximal Regularity of Equations Governed by Non-Autonomous Forms}

\begin{abstract}
	An old problem due to J.-L.~Lions going back to the 1960s asks whether the abstract Cauchy problem associated to non-autonomous symmetric forms has maximal regularity if the time dependence is merely assumed to be continuous or even measurable. We give a negative answer to this question and discuss the minimal regularity needed for positive results.
\end{abstract}

\author{Stephan Fackler}
\address{Institute of Applied Analysis, University of Ulm, Helmholtzstr.\ 18, 89069 Ulm}
\email{stephan.fackler@uni-ulm.de}
\thanks{This work was supported by the DFG grant ``Regularität evolutionärer Probleme mittels Harmonischer Analyse und Operatortheorie''. The author is grateful to the anonymous referees for the careful reading of the manuscript and to H.~Vogt for showing him that the counterexamples even have Hölder exponent $\frac{1}{2}$.}
\keywords{non-autonomous maximal regularity, non-autonomous forms}
\subjclass[2010]{Primary 35B65; Secondary 47A07.}

\maketitle

\section{Introduction}

	Let $(A(t))_{t \in [0,T]}$ for some $T \in (0, \infty)$ be a family of negative generators of $C_0$-semigroups on some fixed  Banach space $X$. One then may consider the following non-autonomous Cauchy problem:
	\begin{equation*}
		\LeftEqNo
		\label{eq:nacp}\tag{NACP}
		\left\{
		\begin{aligned}
			\dot{u}(t) + A(t)u(t) & = f(t) \\
			u(0) & = u_0.
		\end{aligned}
		\right.
	\end{equation*}
	The problem of \emph{$L^p$-maximal regularity} for such a family of operators $(A(t))_{t \in [0,T]}$ and some $p \in (1, \infty)$ asks whether for some space of initial values and all $f \in L^p([0,T];X)$ the problem~\eqref{eq:nacp} has a unique solution in the \emph{maximal regularity space}
		\begin{align*}
			\MaxReg_p([0,T]) \coloneqq \{ & u \in W^{1,p}([0,T];X): u(t) \in D(A(t)) \text{ for almost all } t \in [0,T] \\
			& \text{ and } A(\cdot)u(\cdot) \in L^p([0,T];X) \}.
		\end{align*}
	Here $W^{1,p}([0,T];X)$ is a vector-valued Sobolev space. The term maximal regularity stems from the fact that both summands on the left hand side of the equation~\eqref{eq:nacp} have the highest regularity one can expect, namely the same as the given right hand side. The problem of maximal regularity has attracted major attention in the past decades and has proven to be a fundamental tool for the treatment of non-linear partial differential equations, for an overview see~\cite{Pru02}. The study of maximal regularity on Banach spaces has lead to a sophisticated theory of vector-valued harmonic analysis which by now is known to be crucial for the understanding of the non-Hilbert space case, see the monographs~\cite{DHP03} and~\cite{KunWei04}. In contrast, the Hilbert space case can be treated with more elementary means and is the easiest to study. In this article we will solely focus on this case.
	
	Essentially, there are two different situations. In the autonomous case, i.e.\ $A(t) = A$, a full characterization of maximal regularity was given by L.~de Simon~\cite[Lemma~3,1]{Sim64}: one has maximal regularity if and only if $-A$ generates an analytic $C_0$-semigroup. The non-autonomous case to be considered here is less understood. A particularly clear formulation of the problem of non-autonomous maximal regularity can be given in the special case where all $A(t)$ arise from sesquilinear forms with a fixed domain. 
	
	Let us explain this setting. We will assume that all Hilbert spaces are complex and that sesquilinear forms are linear in the first and anti-linear in the second component. If $V$ is a Hilbert space, we denote by $V'$ its anti-dual, i.e.\ the space of all continuous anti-linear functionals on $V$. Suppose that $H$ is a second Hilbert space in which $V$ densely injects. Then $h \mapsto [ v \mapsto (h|v)_H ]$ induces a dense embedding $H \hookrightarrow V'$ that is compatible with the identification $H \simeq H'$ given by the Riesz representation theorem. Altogether, one obtains a \emph{Gelfand triple} $V \hookrightarrow H \hookrightarrow V'$.
	
	Now let $a\colon [0,T] \times V \times V \to \IC$ be a non-autonomous sesquilinear form, i.e.\ $a(t,\cdot,\cdot)\colon V \times V \to \IC$ is sesquilinear for all $t \in [0,T]$. Then one can associate to each $t \in [0,T]$ an unbounded operator $A(t)$ defined as
	\begin{align*}
		D(A(t)) & = \{ u \in V: \exists f \in H \text{ with } a(t,u,v) = (f|v)_H \text{ for all } v \in V \}, \\
		A(t)u & = f.
	\end{align*} 
	Further, the action of $A(t)$ can be extended to a bounded operator $\mathcal{A}(t)\colon V \to V'$ via $\langle \mathcal{A}(t)u, v \rangle = a(t,u,v)$. Alternatively, $\mathcal{A}(t)$ can be seen as an unbounded operator on $V'$. Then $A(t)$ is the part of $\mathcal{A}(t)$ in $H$. Using this extension one obtains the following weaker form of the Cauchy problem~\eqref{eq:nacp}:
	\begin{equation*}
		\LeftEqNo
		\label{eq:wacp}\tag{WACP}
		\left\{
		\begin{aligned}
			\dot{u}(t) + \mathcal{A}(t)u(t) & = f(t) \\
			u(0) & = u_0.
		\end{aligned}
		\right.
	\end{equation*}
	We will make the following standard assumptions on the forms: there exist constants $\alpha, M > 0$ such that for all $u, v \in V$ and $t \in [0,T]$ one has the estimates
	\begin{equation}
		\tag{A}
		\label{eq:form_assumptions}
		\begin{aligned}
			\abs{a(t,u,v)} & \le M \norm{u}_V \norm{v}_V \\
			\Re a(t,u,u) & \ge \alpha \norm{u}_V^2.
		\end{aligned}
	\end{equation}
	The first assumption means that the forms are uniformly bounded, whereas the second is a uniform \emph{coercivity} assumption. Under these assumptions one has the following existence and uniqueness result for the problem~\eqref{eq:wacp} due to J.-L.~Lions~\cite[p.~513, Theorem~2]{DauLio92}.
	
	\begin{theorem}[Lions]\label{thm:lions}
		Let $a\colon [0,T] \times V \times V \to \IC$ be a non-autonomous sesquilinear form satisfying~\eqref{eq:form_assumptions} which is strongly measurable, i.e.\ for all $u, v \in V$ the function $a(\cdot, u, v)$ is measurable. Then for all $f \in L^2([0,T];V')$ and all $u_0 \in H$ problem~\eqref{eq:wacp} has a unique solution in $L^2([0,T];V) \cap H^1([0,T];V')$.
	\end{theorem}
	
	Here $H^1([0,T];V')$ denotes the vector-valued Sobolev space with square-integrable first derivative. This result is extremely satisfying from a theoretical point of view in that it requires minimal regularity on the form. Note that for given $f \in L^2([0,T];H)$ and, say $u_0 = 0$, it is not clear whether the solution $u$ given by Theorem~\ref{thm:lions} lies in $\MaxReg_2([0,T])$. In fact, this is the case if and only if $\mathcal{A}(\cdot)u(\cdot)$ or $\dot{u}$ and then automatically both lie in $L^2([0,T];H)$. The fact that $u$ lies in $\MaxReg_2([0,T])$ is important for applications as for example boundary conditions tend to be only encoded in $A(t)$, but not in its extension $\mathcal{A}(t)$.
	
	Therefore much effort has been put in the question of maximal regularity of equations governed by forms. We now give a non-complete historical overview of the obtained positive results. Lions obtained maximal regularity under one of the following two sets of additional assumptions~(\cite[p.~65, Théorème~6.1]{Lio61} and~\cite[p.~94, Théorème~6.1]{Lio61}).
	\begin{thm_enum}
		\item $u_0 = 0$, $a(\cdot, u,v) \in C^1([0,T])$ for all $u, v \in V$ and $a$ is \emph{symmetric}, i.e.\ $a(t,u,v) = \overline{a(t,u,v)}$ for all $t \in [0,T]$ and $u, v \in V$.
		\item $u_0 \in D(A(0))$, $a(\cdot, u,v) \in C^2([0,T])$ for all $u, v \in V$ and $f \in H^1([0,T];H)$.
	\end{thm_enum}
	Recently, major progress was made in the following results, the first one due to Haak and Ouhabaz~\cite[Theorem~2]{HaaOuh15} (for a Banach space variant see~\cite[Theorem~4.5]{Fac15c}), the second due to Dier and Zacher~\cite[Corollary~1.1]{DieZac16} and the third due to Dier~\cite[Theorem~4.1]{Die15}.
	\begin{thm_enum}[resume]
		\item\label{item:hoelder} $u_0 \in [D(A(0)), H]_{\frac{1}{2}}$ and $a$ is $\alpha$-Hölder continuous for some $\alpha > \frac{1}{2}$, i.e.\ there exists $K > 0$ with $\abs{a(t,u,v) - a(s,u,v)} \le K \abs{t-s}^{\alpha} \norm{u}_V \norm{v}_V$ for all $t \in [0,T]$ and all $u, v \in V$.
		\item $u_0 \in [D(A(0)), H]_{\frac{1}{2}}$ and one has the fractional Sobolev regularity $\mathcal{A}(\cdot) \in W^{\alpha, 2}([0,T]; \mathcal{B}(V,V'))$ for some $\alpha > \frac{1}{2}$.
		\item\label{item:bv} $u_0 \in V$, $a$ is symmetric and has bounded variation, i.e.\ there exists a bounded and non-decreasing function $g\colon [0,T] \to \IR_{\ge 0}$ with $\abs{a(t,u,v) - a(s,u,v)} \le (g(t) - g(s)) \norm{u}_V \norm{v}_V$ for all $0 \le s \le t \le T$ and all $u, v \in V$.
	\end{thm_enum}
	It is known from the autonomous case that the largest possible space for the initial data is $[D(A(0)), H]_{1/2} = (D(A(0)), H)_{1/2,2}$, see~\cite[Lemma~1]{ChiFio14} and~\cite[Proposition~1.2.10]{Lun95} for the identification of the space with the real interpolation space and~\cite[Remark~3.6]{ChaHewMoi15} for the coincidence of the complex with the real interpolation method on Hilbert couples. If $A(0)$ satisfies the \emph{Kato square root property}, one has $[D(A(0)), H]_{1/2} = V$. For details on this property we refer to~\cite[Section~5.5]{Are04}. In particular, $A(0)$ has the Kato square root property if $a(0,\cdot,\cdot)$ is symmetric. Hence, the results~\ref{item:hoelder}-\ref{item:bv} obtain the optimal space for the initial data. Albeit these partial positive answers, the following question asked by J.-L.~Lions in his 1961 monograph~\cite[p.~68]{Lio61} has still been open.
	
	\begin{problem}[Lions' Problem]\label{problem:lions}
		Let $a\colon [0,T] \times V \times V \to \IC$ be a strongly measurable or, alternatively, continuous symmetric non-autonomous sesquilinear form satisfying~\eqref{eq:form_assumptions}. Does then for $u_0 = 0$ and all $f \in L^2([0,T];H)$ the solution $u$ given by Theorem~\ref{thm:lions} satisfy $u \in \MaxReg_2([0,T])$?
	\end{problem}
	
	The importance of the problem stems from the fact that a positive solution -- in particular for the measurable case -- would have profound applications in the study of quasilinear parabolic problems. In fact, maximal regularity only assuming measurability would allow the application of fixed point theorems without any additional a priori information on the regularity of the solution or the coefficients.
	
	In the negative direction Dier has recently shown in his thesis~\cite[Section~5.2]{Die14} that Problem~\ref{problem:lions} has a negative answer for discontinuous forms if one omits the requirement that the forms are symmetric. The counterexample crucially relies on the fact that there exist forms which do not have the Kato square root property, a sophisticated result due to A.~McIntosh~\cite{McI72}. In particular, it is not applicable to symmetric forms and to second order elliptic differential operators in divergence form as a consequence of the celebrated positive solution of Kato's conjecture given in~\cite{AHLM+02}. This leaves open a central question: is Kato's square root property the only additional property needed and does therefore Problem~\ref{problem:lions} even have a positive answer if one, more generally, assumes that $a$ satisfies the Kato square root property in a uniform sense -- which is the case for non-autonomous second order elliptic operators in divergence form -- or is there a fundamental obstruction to non-autonomous maximal regularity caused by rough time dependencies? 
	
	We will show that the latter is the case. In the main result (Theorem~\ref{thm:ce_symmetric}) we answer both forms of Lions' Problem in the negative. Furthermore, we tackle the problem of the minimal regularity needed for positive results. In fact, Theorem~\ref{thm:ce_symmetric} more precisely shows that the order of differentiability in the positive results is optimal: our counterexample is $\frac{1}{2}$-Hölder continuous. However, positive results with less regularity can be obtained if the domains of the operators $A(t)$ are more regular, see~\cite[Theorem~1.1]{GalVer14}.
	
\section{The General Strategy}
	
	In this section we outline the general strategy behind our approach to the counterexamples. Let $u \in L^2([0,T];V) \cap H^1([0,T];V')$ be given. Assume that $u$ solves the Cauchy problem~\eqref{eq:wacp} for some $f \in L^2([0,T];H)$. Then we have
	\[
		\langle \dot{u}(t), v \rangle_{V',V} + a(t, u(t), v) = (f(t)|v)_H \qquad \text{for all } v \in V \text{ and } t \in [0,T].
	\]
	Here the solution prescribes the values of the form $a(t, \cdot, \cdot)$ on the set $\langle u(t) \rangle \times V$ as
	\begin{equation}
		\label{eq:the_form}
		a(t, c \cdot u(t), v) = c \left[ (f(t)|v)_H - \langle \dot{u}(t), v \rangle_{V',V} \right]
	\end{equation}
	for all $c \in \IC$ and $v \in V$. Our idea is not to start with the form $a$, but to start with a bad behaved function $u$ for which we then try to find a suitable form satisfying~\eqref{eq:the_form}.
	
	Requiring our standard assumptions~\eqref{eq:form_assumptions} on the form of course implies restrictions on $f$ and $u$. In particular, for $a$ to be uniformly coercive we need for all $t \in [0,T]$
	\begin{equation}
		\label{eq:ansatz_coercive}
    	\begin{aligned}
    		\Re a(t, u(t), u(t)) & = \Re (f(t)|u(t))_H - \Re \langle \dot{u}(t), u(t) \rangle_{V',V} \\
    		& = \Re (f(t)|u(t))_H - \frac{1}{2} \frac{d}{dt} \norm{u(t)}_H^2 \ge \alpha \norm{u}_V^2.
    	\end{aligned}
	\end{equation}
	A proof of the differentiation rule used in the second equality can be found in~\cite[Chapter~III, Proposition~1.2]{Sho97}. A first ansatz to satisfy the above estimate is to require that the time derivative of $\norm{u(t)}_H^2$ vanishes. For this we now switch to a concrete setting. We choose $H = L^2([0,1])$ and $V = L^2([0,1]; w \d\lambda)$ for some measurable locally bounded weight $w\colon [0,1] \to \IR_{\ge 1}$. Then $V' = L^2([0,1]; w^{-1} \d\lambda)$ and the natural inclusion gives rise to a Gelfand triple $V \hookrightarrow H \hookrightarrow V'$. A natural choice that makes $\norm{u(t)}_H^2$ time independent is $u(t,x) = c(x) \exp(i t \phi(x))$ for measurable functions $c\colon [0,1] \to \IR$ and $\phi\colon [0,1] \to \IR_{\ge 1}$. In fact,
	 \[
	 	\int_0^1 \abs{u(t,x)}^2 \d x = \int_0^1 \abs{c(x)}^2 \d x = \norm{c}_H^2.
	 \]
	 Of course, such a function will not be the solution of~\eqref{eq:wacp} for a symmetric form. Nevertheless we will see that this choice gives us a good starting point that reduces some technical details. Observe that $u \in L^2([0,T];V) \cap H^1([0,T];V')$ solving~\eqref{eq:wacp} has maximal regularity in $H$ if and only if the distributional derivative satisfies $\dot{u} \in L^2([0,T];H)$. Hence, we require for the time derivative $\dot{u}(t,x) = i \phi(x) c(x) \exp(it \phi(x))$ the validity of
	 \begin{equation*}
	 	\tag{H}
		\label{eq:H}
	 	\int_0^1 \abs{\dot{u}(t,x)}^2 \d x = \int_0^1 \abs{\phi(x) c(x)}^2 \d x = \norm{\phi c}_H^2 = \infty
	 \end{equation*}
	 besides the integrability conditions $u \in L^2([0,T];V)$ and $\dot{u} \in L^2([0,T];V')$ given by 
	 \begin{align}
	 	\int_0^1 \abs{u(t,x)}^2 w(x) \d x & = \int_0^1 \abs{c(x)}^2 w(x) \d x = \norm{c}_V^2 < \infty, \label{eq:V}\tag{V} \\
		\int_0^1 \abs{\dot{u}(t,x)}^2 w^{-1}(x) \d x & = \int_0^1 \abs{\phi(x) c(x)}^2 w^{-1}(x) \d x = \norm{\phi c}_{V'}^2 < \infty. \label{eq:V'}\tag{V'}
	 \end{align}
	 Coming back to inequality~\eqref{eq:ansatz_coercive}, we obtain for our ansatz and the choice $f(t) = u(t)$
	 \begin{align*}
	 	\Re a(t, u(t), u(t)) & = \Re (f(t)|u(t))_H = \norm{u(t)}_H^2 = \norm{c}_H^2 = \alpha \norm{u(t)}_V^2,
	 \end{align*}
	 where $\alpha$ is the quotient $\norm{c}^2_H / \norm{c}^2_V$. Further, the boundedness of~\eqref{eq:the_form} on $\langle u(t) \rangle \times V$ follows from $u \in L^{\infty}([0,T]; H)$ and $\dot{u} \in L^{\infty}([0,T];V')$. In order to obtain a counterexample we need to extend these sesquilinear mappings to $V \times V$.
	 
\section{Extending Forms}

	We now study the problem of extending partially defined sesquilinear forms in an abstract setting. Let $V$ be a complex Hilbert space and $U \subset V$ a closed subspace. Suppose one is given a partially defined bounded sesquilinear form $b\colon U\times V \to \IC$. For fixed $u \in V$ one has $b(u, \cdot) \in V'$ and therefore by the Riesz representation theorem there exists a unique element $Tu \in V$ with $b(u,v) = (Tu|v)_V$ for all $v \in V$. This gives rise to a bounded linear operator $T\colon U \to V$ with $\norm{T} = \norm{b}$, where $\norm{b}$ is the supremum of all $b(u,v)$ with $u \in U$ and $v \in V$ of norm at most one. Using this one-to-one correspondence between bounded sesquilinear forms and bounded operators, the bounded sesquilinear extensions $a\colon V \times V \to \IC$ of $b$ correspond to bounded extensions $\hat{T}\colon V \to V$ of $T$. Clearly, unless $U = V$, there exist uncountable many such extensions. If further $b$ is assumed to be \emph{accretive}, which means
		\[
			\Re b(u,u) \ge 0 \qquad \text{for all } u \in U,
		\]
	one is interested in accretive extensions. We have $b(u,u) = (Tu|u)_V$ for all $u \in U$ and therefore we search  extensions $\hat{T}\colon V \to V$ preserving certain properties of the numerical range of $T$ defined in the subspace setting considered here as
		\[
			W(T) \coloneqq \{ (Tu|u)_V: u \in U, \norm{u}_V = 1 \}.
		\]
		Recall that in the standard case $V = U$ the operator $T\colon V \to V$ is selfadjoint if and only if $W(T) \subset \IR$~\cite[Theorem~1.2-2]{GusRao97}. 
		
	The extension problem has a straightforward solution if $T$ leaves $U$ invariant.
	
	\begin{lemma}\label{lem:invariant_extension}
		Let $U \subset V$ be a closed subspace of a Hilbert space $V$ and $T\colon U \to U$ linear and bounded with $W(T) \subset \{ z \in \IC: \Re z \ge 0\}$. Then there exists an extension $\hat{T}\colon V \to V$ of $T$ with $\normalnorm{\hat{T}} = \norm{T}$ and $W(\hat{T}) \subset \{ z \in \IC: \Re z \ge 0\}$. If $W(T) \subset \IR$, then $\hat{T}$ can be chosen self-adjoint.
	\end{lemma}
	\begin{proof}
		Using the orthogonal complement $U^{\perp}$ of $U$ in $V$, we define the trivial extension
			\[
				\hat{T}\colon U \oplus U^{\perp} \ni u + u^{\perp} \mapsto Tu.
			\]
		Notice that $\normalnorm{\hat{T}} = \norm{T}$ and that $W(\hat{T})$ is the convex hull of $W(T)$ and $\{ 0 \}$, which is a subset of $\{ z \in \IC: \Re z \ge 0\}$. Clearly, if $W(T) \subset \IR$, then $W(\hat{T}) \subset \IR$.
	\end{proof}
	
	We now come to the interesting case in which $T$ does not leave the subspace $U$ invariant. Recall that in the situation of the previous section $U$ is one-dimensional.
	
	\begin{proposition}\label{prop:range_extension_1d}
		Let $V^{\IR}$ be a real Hilbert space and $V$ its complexification. Further, let $U \subset V$ be a one-dimensional subspace of $V$ and $T\colon U \to V$ bounded and linear. 
		\begin{thm_enum}
			\item\label{prop:range_extension_1d:a} If $W(T) \subset \{ z \in \IC: \Re z \ge 0\}$, then there exists an extension $\hat{T}\colon V \to V$ of $T$ with $\normalnorm{\hat{T}} \le \sqrt{2} \norm{T}$ and $W(\hat{T}) \subset \{ z \in \IC: \Re z \ge 0 \}$. 
			\item\label{prop:range_extension_1d:b} Let $\epsilon > 0$. If $U \cap V^{\IR} \neq 0$, $T(U \cap V^{\IR}) \subset V^{\IR}$ and $W(T) \subset \{ z \in \IC: \Re z \ge \epsilon \}$, then there exists a self-adjoint extension $\hat{T}\colon V \to V$ of $T$ with $\normalnorm{\hat{T}} \le \sqrt{2} (\norm{T} + \epsilon^{-1} \norm{T}^2)$ and $W(\hat{T}) \subset \IR_{\ge 0}$.
		\end{thm_enum}
	\end{proposition}
	\begin{proof}
		Let $W = \linspan \{ U, TU \}$. Then $\dim W \le 2$. If $\dim W = 1$, then $TU \subset U$ and the result follows from Lemma~\ref{lem:invariant_extension}. Hence, we may assume that $\dim W = 2$. 
		
		For part~\ref{prop:range_extension_1d:a} choose an orthonormal basis $(e_1, e_2)$ of $W$ with $e_1 \in U$. Let $S\colon W \to W$ be an extension of $T$. For $w = \lambda_1 e_1 + \lambda_2 e_2 \in W$ we then have
			\begin{align*}
				(Sw|w) = \abs{\lambda_1}^2 (Te_1|e_1) + \abs{\lambda_2}^2 (Se_2|e_2) + \lambda_1 \overline{\lambda_2} (Te_1 | e_2) + \overline{\lambda_1} \lambda_2 (Se_2|e_1).
			\end{align*}
		Choose an extension $S$ satisfying $(Se_2|e_1) = - \overline{(Te_1|e_2)}$. This gives us
			\[
				\Re (Sw|w) = \abs{\lambda_1}^2 \Re (Te_1|e_1) + \abs{\lambda_2}^2  \Re (Se_2|e_2).
			\]
		Hence, $S$ has the desired properties if we choose $(Se_2 | e_2) = 0$. This choice implies
			\[
				\norm{Sw} \le (\abs{\lambda_1} + \abs{\lambda_2}) \norm{T} \le \sqrt{2} \norm{T} \norm{w}.
			\]
		In the setting of part~\ref{prop:range_extension_1d:b} we can find an orthonormal basis $(e_1, e_2)$ of $W$ in $V^{\IR}$ with $e_1 \in U \cap V_{\mathbb{R}}$. We now choose $(Se_2|e_1) = (Te_1|e_2)$, which is a real number by the made assumptions. Then for $(Se_2|e_2) \ge 0$ we obtain
			\[
				(Sw|w) = \abs{\lambda_1}^2 (Te_1| e_1) + \abs{\lambda_2}^2 (Se_2| e_2) + (\lambda_1 \overline{\lambda_2} + \overline{\lambda_1} \lambda_2) (Te_1| e_2) \in \IR.
			\]
		Further, we have
			\begin{align*}
				\MoveEqLeft (Sw|w) \ge \abs{\lambda_1}^2 (Te_1| e_1) + \abs{\lambda_2}^2 (Se_2| e_2) - 2 \abs{\lambda_1} \abs{\lambda_2} \abs{(Te_1|e_2)} \\
				& \ge \bigg[(Se_2|e_2) - \frac{\abs{(Te_1|e_2)}^2}{(Te_1|e_1)} \biggr] \abs{\lambda_2}^2 \ge [(Se_2|e_2) - \epsilon^{-1} \norm{T}^2 ] \abs{\lambda_2}^2.
			\end{align*}
		The right hand side is non-negative if we choose $(Se_2|e_2) = \epsilon^{-1} \norm{T}^2$. For the operator norm of this extension we get
			\begin{align*}
				\norm{Sw} & \le \abs{\lambda_1} \norm{Te_1} + \abs{\lambda_2} \norm{S e_2} \le \abs{\lambda_1} \norm{T} + \abs{\lambda_2} (\norm{T} + \epsilon^{-1} \norm{T}^2) \\
				& \le \sqrt{2} (\norm{T} + \epsilon^{-1} \norm{T}^2) \norm{w}.
			\end{align*}
		Since in both parts the constructed operator $S$ leaves $W$ invariant, we can now apply Lemma~\ref{lem:invariant_extension} to $S$ and obtain the desired extension $\hat{T}$ of $S$.
	\end{proof}
	
	In terms of sesquilinear forms we obtain the following extension result.
	
	\begin{proposition}\label{prop:extension_form}
		Let $V$ be the complexification of a real Hilbert space $V^{\mathbb{R}}$, $U \subset V$ a one-dimensional subspace of $V$ and $b\colon U \times V \to \IC$ a sesquilinear form such that for some constants $\alpha > 0$ and $M > 0$ one has for all $u \in U$ and $v \in V$,
			\begin{equation*}
				\abs{b(u,v)} \le M \norm{u}_V \norm{v}_V, \qquad \Re b(u,u) \ge \alpha \norm{u}_V^2.
			\end{equation*}
		Then $b$ can be extended to a sesquilinear form $a\colon V \times V \to \IC$ with
			\begin{align*}
				\abs{a(u,v)} & \le \left[ \sqrt{2} \left(M + \frac{\alpha}{2} + 2\alpha^{-1} \left(M + \frac{\alpha}{2} \right)^2 \right) + \frac{\alpha}{2} \right] \norm{u}_V \norm{v}_V, \\
				 \Re a(u,u) & \ge \frac{\alpha}{2} \norm{u}_V^2
			\end{align*}
		for all $u, v \in V$. If additionally there exists $0 \neq u \in U \cap V^{\IR}$ with $a(u,v) \in \IR$ for all $v \in V^{\IR}$, then $a$ can be chosen to be symmetric.
	\end{proposition}
	\begin{proof}
		Let $T\colon U \to V$ be the operator associated to the form $b(\cdot,\cdot) - \frac{\alpha}{2} (\cdot|\cdot)_V$. Note that the operator $T + \frac{\alpha}{2} \Id_U$ is associated to the form $b$. Let $\hat{T}$ be the extension of $T$ given by Proposition~\ref{prop:range_extension_1d} and $a$ the operator associated to $\hat{T} + \frac{\alpha}{2} \Id_V$. Then $a$ extends $b$ and has the desired properties.
	\end{proof}
	
\section{Time Regularity of the Extended Forms and a Non-Symmetric Counterexample}\label{sec:non_symmetric}

	We can now use Proposition~\ref{prop:extension_form} to extend shifts of the forms $b(t,\cdot,\cdot)\colon \langle u(t) \rangle \times V \to \IC$ defined by~\eqref{eq:the_form} to forms $a(t,\cdot,\cdot)\colon V \times V \to \IC$. Here it is crucial to know the time regularity of the mappings $t \mapsto a(t,u,v)$ for $u, v \in V$. For our concrete ansatz $u(t,x) = c(x) \exp(i t \phi(x))$ and $f(t) = u(t)$ we obtain
		\begin{align*}
			b(t, u(t), v) & = (u(t)|v)_H - \langle \dot{u}(t), v \rangle_{V',V} = \int_0^1 u(t,x) \overline{v}(x) \d x - \int_0^1 \dot{u}(t,x) \overline{v}(x) \d x \\
			& = (w^{-1} (u(t) - \dot{u}(t)) | v)_V.
		\end{align*}
	Note that $b(t,u(t),u(t)) \not\in \IR$ and therefore $b$ has no symmetric extensions. However, we have the advantage that Proposition~\ref{prop:extension_form} can use the easier first part of Proposition~\ref{prop:range_extension_1d}. Going back to the proofs of~Propositions~\ref{prop:extension_form} and~\ref{prop:range_extension_1d}, fixing $t \in [0,T]$ and using the same notation (except for an additional subscript indicating the time dependence) we have $e_1 = u(t) / \norm{u(t)}_V$ and get for the operator $T_t$ associated to $b(t,\cdot,\cdot) - \frac{\alpha}{2} (\cdot|\cdot)_V$
		\[
			T_t e_1 = \frac{u(t) - \dot{u}(t)}{w \norm{u(t)}_V} - \frac{\alpha}{2} \frac{u(t)}{\norm{u(t)}_V}.
		\]
	Further, the part of $T_t u(t)$ orthogonal to $e_1$ is
		\begin{equation}
			\label{eq:formula_z}
			\begin{split}
    			\MoveEqLeft T_t u(t) - (T_t u(t)|e_1)_V e_1 = T_t u(t) - \left[ -\frac{\alpha}{2} + \frac{1}{\norm{u(t)}_V^2} \biggl( \norm{u(t)}_H^2 - \langle \dot{u}(t), u(t) \rangle \biggr) \right] u(t) \\
    			& =  -w^{-1} \dot{u}(t) + \biggl[w^{-1} - \frac{1}{\norm{u(t)}_V^2} \biggl( \norm{u(t)}_H^2 - \langle \dot{u}(t), u(t) \rangle \biggr) \biggr] u(t) \eqqcolon z(t).
			\end{split}
		\end{equation}
	Although the above expression is lengthy, its norm in $V$ is easily seen to be time independent for our concrete choice of $u$. In particular, there exist fixed normalization factors $n_1 = \norm{u(t)}_V$ and $n_2 = \norm{z(t)}_V$ for $u(t)$ and $z(t)$ in $V$. Further, the orthogonal projection $P_t\colon V \to \linspan \{ u(t), T_t u(t) \}$ is given by
		\begin{align*}
			P_t v = \frac{1}{\norm{u(t)}_V^2} (v|u(t))_V u(t) + \frac{1}{\norm{z(t)}_V^2} (v|z(t))_V z(t).
		\end{align*}
	One then has for $v_1,v_2 \in V$
		\begin{equation}
			\label{eq:a_two_dimensional}
			\begin{split}
				\MoveEqLeft (S_t P_t v_1| P_t v_2)_V = \frac{1}{\norm{u(t)}_V^4} (v_1|u(t))_V \overline{(v_2|u(t))_V} (T_t u(t)|u(t))_V \\
				& + \frac{1}{\norm{u(t)}_V^2 \norm{z(t)}_V^2} (v_1|u(t))_V \overline{(v_2|z(t))_V} (T_t u(t)|z(t))_V \\
				& - \frac{1}{\norm{u(t)}_V^2 \norm{z(t)}_V^2} (v_1|z(t))_V \overline{(v_2|u(t))_V} \overline{(T_t u(t)|z(t))_V}.
			\end{split}
		\end{equation}
	Putting everything together, the extended forms are given for $v_1, v_2 \in V$ by
		\begin{equation}
			\label{eq:a_whole}
			\begin{split}
				a(t,v_1,v_2) & = (\hat{T}_t v_1 | v_2)_V + \frac{\alpha}{2}(v_1 | v_2)_V = (S_t P_t v_1| P_t v_2)_V + \frac{\alpha}{2}(v_1 | v_2)_V.
			\end{split}
		\end{equation}
	In the following we are interested in the regularity of these mappings.
	
\subsection{Hölder continuity}

	For the Hölder continuity we make use of the fact that the (scalar) product of two bounded Hölder continuous functions is Hölder continuous of the same exponent. Taking this into account, we see from the explicit form given in~\eqref{eq:a_two_dimensional} and~\eqref{eq:a_whole} that $a$ is $\alpha$-Hölder continuous if $u\colon [0,T] \to V$, $w^{-1} u \colon [0,T] \to V$ and $\dot{u} \colon [0,T] \to V'$ are $\alpha$-Hölder continuous. We now deal with all three functions. For the first one has for $s, t \in [0,T]$
	\begin{align*}
		\MoveEqLeft \norm{u(t) - u(s)}_V^2 = \int_0^1 w(x) \abs{u(t,x) - u(s,x)}^2 \d x \\
		& = \int_0^1 w(x) \abs{c(x)}^2 \abs{\exp(it \phi(x)) - \exp(is \phi(x))}^2 \d x.
	\end{align*}
	We now optimize the Hölder exponent of the above expression under the constraint that all conditions \eqref{eq:V}, \eqref{eq:V'} and \eqref{eq:H} are satisfied. In the class of functions of the form $w(x) = x^{-a}$, $\phi(x) = x^{-b}$ and $c(x) = x^c$ for $a, b, c \ge 0$, a maximum for the Hölder exponent is obtained for the choices $a = b = \frac{3}{2}$ and $c = 1$. One can immediately verify that for this choice \eqref{eq:V}, \eqref{eq:V'} and \eqref{eq:H} are satisfied. Further, it follows from the estimates $\abs{\exp(i t \phi(x))} \le 1$ and $\abs{\exp(it\phi(x))-\exp(is \phi(x))} \le \abs{t-s} \phi(x)$ that for $\abs{t-s} \le 1$
	\begin{align*}
		\norm{u(t)-u(s)}_V^2 & \le 4 \int_0^{\abs{t-s}^{2/3}} x^{1/2} \d x + \abs{t-s}^2 \int_{\abs{t-s}^{2/3}}^1 x^{1/2} x^{-3} \d x \\
		& \le \frac{8}{3} \abs{t-s} + \frac{2}{3} \abs{t-s}^2 \abs{t-s}^{-1} = \frac{10}{3} \abs{t-s},
	\end{align*}
	whereas for $\abs{t-s} \ge 1$ the same Hölder estimate is trivial. This shows that $u\colon [0,T] \to V$ is $\frac{1}{2}$-Hölder continuous. The two remaining functions have the same Hölder exponents: one has $\norm{w^{-1} (u(t) - u(s))}_V \le \norm{u(t) - u(s)}_V$, whereas for $\dot{u}\colon [0,T] \to V'$ one obtains for our concrete choice of functions the same expression as for $u\colon [0,T] \to V$. Putting everything together, we obtain the following counterexample.
	
	\begin{theorem}\label{thm:ce_non-symmetric}
		For all $T \in (0,\infty)$ there exist a Gelfand triple $V \hookrightarrow H \hookrightarrow V'$ and a form $a\colon [0,T] \times V \times V \to \IC$ satisfying~\eqref{eq:form_assumptions} such that for some $K \ge 0$ one has
			\[
				\abs{a(t,u,v) - a(s,u,v)} \le K \abs{t-s}^{\frac{1}{2}} \norm{u}_V \norm{v}_V \qquad \text{for all } t, s \in [0,T], u,v \in V
			\]
		and for which the associated problem~\eqref{eq:nacp} fails to have maximal regularity for the initial value $u_0 = 0$ and some inhomogenity in $L^2(0,T;V)$. 
	\end{theorem}
	\begin{proof}
		Let $u$ be as above. Choose a smooth cut-off function $\phi\colon [0,T] \to \IR$ with $0 \le \phi \le 1$, $\phi(0) = 0$ and $\phi(t) = 1$ for all $t \ge \frac{T}{2}$. Now set $w = \phi u$. Then $w$ satisfies $w(0) = 0$ and $w \in H^1([0,T];V') \cap L^2([0,T];V)$ as well as
		\begin{align*}
			\dot{w}(t) + \mathcal{A}(t)w(t) = \phi(t) \dot{u}(t) + \phi(t) \mathcal{A}(t) u(t) +  \dot{\phi}(t) u(t) = \phi(t) u(t) + \dot{\phi}(t) u(t).
		\end{align*}
		The right hand side of the equation lies in $L^2([0,T];V)$ because $u$ does. Further, on $(T/2,T]$ one has $\dot{w} = \dot{u}$ and therefore $\dot{w}(t) \not\in H$ and a fortiori $\dot{w} \not\in L^2([0,T];H)$. Hence, $w \not\in \MaxReg_2([0,T])$.
	\end{proof}

\section{A Symmetric Counterexample}\label{sec:symmetric}

	In this section we modify the previous used approach to construct symmetric non-autonomous forms failing maximal regularity. For this we use the ansatz $u(t,x) = c(x) (\sin(t\phi(x)) + d)$ with $d \in \IR$ and $c, \phi$ as before. Except for some shift which will become handy later, this is the imaginary part of the counterexample used before. We again require conditions \eqref{eq:V}, \eqref{eq:V'} and \eqref{eq:H} to hold. Note that the first two still hold if both $\norm{c}_V$ and $\norm{\phi c}_{V'}$ are finite, which we assume from now on. Concerning~\eqref{eq:H} one requires
		\begin{equation}
			\label{eq:the_final_infinity}
			\int_0^1 \abs{\dot{u}(t,x)}^2 \d x = \int_0^1 \abs{\phi(x) c(x)}^2 \cos^2(t\phi(x)) \d x = \infty,
		\end{equation}
	which we will check explicitly for concrete functions later on.
	Since the forms $a(t,\cdot,\cdot)$ should be symmetric, instead of \eqref{eq:ansatz_coercive} we need for the same choice $f(t) = u(t)$
		\begin{equation}
			\label{eq:symmetric_coercive_estimate}
			a(t, u(t), u(t)) = (u(t)|u(t))_H - \langle \dot{u}(t), u(t) \rangle_{V',V} = \norm{u(t)}_H^2 - \frac{1}{2} \frac{d}{dt} \norm{u(t)}_H^2 \ge \epsilon
		\end{equation}
	for some $\epsilon > 0$ independent of $t \in [0,T]$. In fact, one then obtains
		\[
			a(t,u(t),u(t)) \ge \frac{\epsilon}{\norm{u(t)}_V^2} \norm{u(t)}_V^2 \ge \alpha \norm{u(t)}_V^2
		\]
	for some $\alpha > 0$ since $\norm{u(\cdot)}_V^2$ is uniformly bounded from above. We now verify~\eqref{eq:symmetric_coercive_estimate}. Inserting our choice of $u$, one has
		\begin{equation*}
			\norm{u(t)}_H^2 = \int_0^1 \abs{c(x)}^2 \left( \sin(t\phi(x)) + d \right)^2 \d x.
		\end{equation*}
	Differentiating this identity gives
		\begin{equation*}
			\frac{1}{2} \frac{d}{dt} \norm{u(t)}_H^2 =  \int_0^1 \abs{c(x)}^2 \phi(x) \cos(t\phi(x)) \left( \sin(t\phi(x))  + d \right) \d x
		\end{equation*}
		Taking the difference of the two expressions, we get
		\begin{equation}
			\label{eq:establish_one_dimensional_coercive}
			\begin{split}
			\MoveEqLeft \norm{u(t)}_H^2 - \frac{1}{2} \frac{d}{dt} \norm{u(t)}_H^2 \ge \int_0^1 \abs{c(x)}^2 \left[ (\abs{d} - 1)^2 - \abs{\phi(x)} (\abs{d} + 1) \right] \d x \\
			& = (\abs{d} - 1)^2 \int_0^1 \abs{c(x)}^2 \d x - (\abs{d}+ 1) \int_0^1 \abs{c(x)}^2 \abs{\phi(x)} \d x.
			\end{split}
 		\end{equation}
		The first integral is finite by our assumptions on $c$ and for the second we have
		\begin{align*}
			\MoveEqLeft \int_0^1 \abs{c(x)}^2 \abs{\phi(x)} \d x \\
			& \le \left( \int_0^1 \abs{c(x)}^2 w(x) \d x \right)^{1/2} \left( \int_0^1 \abs{c(x) \phi(x)}^2 w^{-1}(x) \d x \right)^{1/2} < \infty.
		\end{align*}
		Hence, the left hand side of~\eqref{eq:establish_one_dimensional_coercive} is positive if $\abs{d} > 1$ is chosen sufficiently large, which we assume from now on. Proposition~\ref{prop:extension_form} yields uniformly bounded and coercive symmetric forms $a(t,\cdot,\cdot)\colon V \times V \to \IC$ that satisfy~\eqref{eq:the_form} on $\langle v(t) \rangle \times V$. As before, we now deal with the Hölder regularity of these forms. For this we observe that in the symmetric case formulas~\eqref{eq:a_two_dimensional} and~\eqref{eq:a_whole} for the extended forms must only be slightly modified. In fact, the only changes involve a change of signs and a missing complex conjugate in the third summand of~\eqref{eq:a_two_dimensional} as well as the fourth term
		\[
			K \frac{1}{\norm{z(t)}_V^4} (v_1|z(t))_V \overline{(v_2|z(t))_V}, 
		\]
		where $K$ is a sufficiently large constant (see the proof of Proposition~\ref{prop:range_extension_1d}). By the same reasoning as before, except for one additional argument, for the Hölder continuity of the form it suffices to deal with the functions~$u\colon [0,T] \to V$, $w^{-1} u \colon [0,T] \to V$ and $\dot{u} \colon [0,T] \to V'$. This argument is needed because the norms of $u(t)$ and $z(t)$ are no longer time independent and therefore the regularity of the reciprocals of their norms must be considered as well. For this we use that if an $\alpha$-Hölder continuous function $f\colon [0,T] \to \IR$ satisfies $\abs{f(t)} \ge \epsilon$ for all $t \in [0,T]$, then $1/f$ is $\alpha$-Hölder continuous as well. For $\norm{u(\cdot)}_V^2$ we have for $\abs{d} > 1$
		\begin{equation*}
			\int_0^1 w(x) \abs{c(x)}^2 \abs{\sin(t \phi(x)) + d}^2 \d x \ge \int_0^1 w(x) \abs{c(x)}^2 (\abs{d} - 1)^2 > 0,
		\end{equation*}
		whereas for $\norm{z(t)}_V^2$ (recall definition~\eqref{eq:formula_z}) we make use of the continuity of $\norm{z(\cdot)}_V^2$ on the compact interval $[0,T]$ together with the fact that $T_tu(t)$ is not a scalar multiple of $u(t)$.
		
		If we choose $c(x) = x$, $w(x) = x^{-3/2}$ and $\phi(x) = x^{-3/2}$ as before, then essentially the same calculations as in Section~\ref{sec:non_symmetric} for the Hölder continuity of the three functions can be used. For example, for $\abs{t-s} \le 1$ we have
			\begin{align*}
				\MoveEqLeft \norm{u(t) - u(s)}_V^2 = \int_0^1 w(x) \abs{u(t,x) - u(s,x)}^2 \d x \\
				& = \int_0^1 w(x) \abs{c(x)}^2 \abs{\sin(t \phi(x)) - \sin(s \phi(x))}^2 \d x \le \frac{10}{3} \abs{t-s}.
			\end{align*}
		Hence, $u$ has the same regularity as in the non-symmetric case considered before. We have arrived at the following negative answer to Lions' problem.
		
	\begin{theorem}\label{thm:ce_symmetric}
		For all $T \in (0,\infty)$ there exist a Gelfand triple $V \hookrightarrow H \hookrightarrow V'$ and a symmetric form $a\colon [0,T] \times V \times V \to \IC$ satisfying~\eqref{eq:form_assumptions} such that for some $K \ge 0$ one has
			\[
				\abs{a(t,u,v) - a(s,u,v)} \le K \abs{t-s}^{\frac{1}{2}} \norm{u}_V \norm{v}_V \qquad \text{for all } t, s \in [0,T], u,v \in V
			\]
		and for which the associated problem~\eqref{eq:nacp} fails to have maximal regularity for the initial value $u_0 = 0$ and some inhomogenity in $L^2(0,T;V)$. 
	\end{theorem}
	\begin{proof}
		We again choose $c(x) = x$, $w(x) = x^{-3/2}$ and $\phi(x) = x^{-3/2}$. By the above considerations the obtained form is $\frac{1}{2}$-Hölder continuous. One can then argue as in the proof of Theorem~\ref{thm:ce_non-symmetric}. Remember that for $u$ not to lie in $\MaxReg_2([0,T])$ the only fact we have not checked yet is the validity of~\eqref{eq:the_final_infinity}. For this we explicitly have
		\begin{equation*}
			\int_0^{1} x^{-1} \cos^2(t x^{-3/2}) \d x = \frac{2}{3} \int_t^{\infty} x^{-1} \cos^2(x) \d x = \infty. \qedhere
		\end{equation*}	
	\end{proof}
	
	\section{Open Problems}
	
	Our findings naturally lead to some further questions. We use the opportunity to formulate some of them explicitly. An important set of questions is motivated by potential applications in PDE. Recall that one central motivation for Lions' Problem (Problem~\ref{problem:lions}) comes from its potential applications to quasilinear parabolic equations.
	
	\begin{problem}\label{problem:lions_elliptic}
		Does Lions' Problem (Problem~\ref{problem:lions}) have a positive solution if the forms induce elliptic operators in divergence form? If not, can the regularity at least be weakened compared to the case of general abstract forms?
	\end{problem}
	
	\begin{remark}
		After the publication of the first preprint version of this article some progress was made on the above problem. Auscher and Egert have shown in the setting of Problem~\ref{problem:lions_elliptic} that one has non-autonomous maximal regularity if the half derivatives of the coefficients have bounded mean oscillation~\cite{AusEge16}. However, the general problem as posed above remains open.
	\end{remark}
				
	\emergencystretch=0.75em
	\printbibliography

\end{document}